\documentclass[a4paper, reqno]{amsart}
\usepackage{amsmath} 
\usepackage{amssymb}
\usepackage{amsthm}
\usepackage{mathrsfs}
\usepackage{mathtools}
\usepackage{graphicx} 
\usepackage{tikz}
\usepackage{tikz-cd}
\usetikzlibrary{commutative-diagrams}
\usepackage{enumitem}
\usepackage{float}
\usepackage{hyperref}
\usepackage{geometry}
\usepackage{adjustbox}
\usepackage{caption}
\usetikzlibrary{calc}
\usepackage[normalem]{ulem}

\theoremstyle{plain}
\newtheorem{theorem}{Theorem}[section]
\newtheorem{corollary}[theorem]{Corollary}
\newtheorem{lemma}[theorem]{Lemma}
\newtheorem{proposition}[theorem]{Proposition}
\theoremstyle{definition}
\newtheorem{definition}[theorem]{Definition}
\newtheorem{remark}[theorem]{Remark}
\newtheorem{example}[theorem]{Example}

\allowdisplaybreaks[2]
\numberwithin{equation}{section}
\numberwithin{figure}{section}
\numberwithin{table}{section}

\begin{document}
\title{Equivariant bordism rigidity for toric manifolds}

\author{Runze Chen}
\address{School of Mathematical Sciences, Fudan University, Shanghai 200433, China}
\email{20110180032@fudan.edu.cn}

\author{Yuanxin Guan}
\address{School of Mathematical Sciences, Fudan University, Shanghai 200433, China}
\email{yxguan21@m.fudan.edu.cn}

\author{Zhi L\"u}
\address{School of Mathematical Sciences, Fudan University, Shanghai 200433, China}
\email{zlu@fudan.edu.cn}

\begin{abstract}
    In this paper, we develop a bordism-theoretic approach to rigidity problems for toric and quasitoric manifolds. 
    We prove that two toric manifolds are isomorphic as varieties if and only if they are weakly equivariantly unitary bordant.
    We also establish a parallel rigidity result for omnioriented quasitoric manifolds satisfying the injectivity condition, showing that their equivariant unitary bordism classes completely determine their omniorientation-preserving equivariant homeomorphism types. 
    Thus, equivariant bordism provides a topological framework for detecting geometric and combinatorial rigidity.
\end{abstract}



\maketitle

\section{Introduction}

Toric varieties are a fundamental class of algebraic varieties, serving as a bridge between algebraic geometry and combinatorics as well as convex geometry. They admit explicit descriptions in terms of rational fans. This combinatorial description allows many geometric and algebraic features of toric varieties to be studied through polyhedral data. 
Consequently, toric varieties have become a natural testing ground for numerous questions in algebraic geometry and a key intersection of algebraic geometry, combinatorics, convex geometry, symplectic geometry, and topology; see \cite{Danilov1978, Oda1988, CLS2011}. 
The close relationship between toric varieties and combinatorics has further inspired topological analogues. In their seminal work \cite{DavisJanus1991}, Davis and Januszkiewicz introduced quasitoric manifolds and small covers, which are now central objects in toric topology; see also \cite{BP_torictopo}. These manifolds provide a particularly natural framework for rigidity questions: to what extent can algebraic, topological, or combinatorial invariants recover the underlying geometric or equivariant structure?

Among the various rigidity problems arising in toric topology, cohomological rigidity has received especially extensive attention. 
Masuda's equivariant cohomological rigidity theorem \cite{Masuda2008}, as corrected in \cite[Remark~2.5(3)]{HKM2022}, states that two toric manifolds are isomorphic as varieties if and only if there exists a weak isomorphism between their equivariant cohomology algebras that preserves their first equivariant Chern classes.
The same holds in the equivariant setting for quasitoric manifolds and small covers: two such manifolds are equivariantly homeomorphic if and only if their equivariant cohomology algebras are isomorphic.
It is therefore reasonable to ask whether the ordinary cohomology ring also determines the underlying topological type. This leads to the cohomological rigidity problem: whether two quasitoric manifolds or two small covers with isomorphic cohomology rings are necessarily homeomorphic or diffeomorphic. 
Partial positive results have been obtained for various classes of generalized Bott manifolds \cite{CMS2010, CS2011, Ishida2012}, quasitoric manifolds \cite{CMS2010_quasi, CPS2010}, and real Bott manifolds \cite{KamishimaMasuda2009}, among others. These results have established cohomological rigidity as one of the central themes in toric topology.

While the preceding results study rigidity from a cohomological perspective, we approach it here from a different viewpoint, namely equivariant bordism. 
More precisely, we ask the following question:
\begin{center}
	\emph{when does the equivariant bordism class determine the isomorphism type?}
\end{center}
Here, by isomorphism we mean isomorphism as algebraic varieties in the toric setting, and equivariant homeomorphism in the topological setting. 
We refer to this question as the \emph{equivariant bordism rigidity property}.
In general, one cannot expect such a rigidity phenomenon, since equivariant bordism is a much weaker equivalence relation than isomorphism. 
Nevertheless, the situation may be quite different for the objects in toric topology. Toric manifolds, quasitoric manifolds, and small covers admit explicit combinatorial descriptions in terms of fans, polytopes, and characteristic data, and their equivariant topology is therefore subject to strong combinatorial constraints. This suggests that, within these special categories, equivariant bordism may retain substantially more geometric information than it does for general equivariant manifolds. 
An important foundation for establishing such a connection is provided by classical fixed-point results in equivariant bordism. Hamrick and Ossa \cite{HamrickOssa} showed that the equivariant unitary bordism class of a $T^k$-manifold is completely determined by the equivariant normal data of its fixed point set. In particular, when the fixed points are isolated, the signed tangent representations at the fixed points determine the equivariant unitary bordism class. 
For the real case, Stong \cite{Stong1970} proved that the tangent representations at the fixed points completely determine the equivariant unoriented bordism class of smooth $\mathbb Z_2^k$-manifolds with finitely many fixed points.
These results provide the bridge between equivariant bordism and the combinatorial data of toric manifolds, quasitoric manifolds and small covers, and form the starting point for the bordism rigidity results established in this paper.

By describing the equality of the equivariant unitary bordism classes and the existence of isomorphisms of toric objects at the combinatorial level, we establish the following bordism rigidity results for three main classes of toric objects.
Our first main result concerns toric manifolds.

\begin{theorem} \label{main_theorem_toric}
	Two toric manifolds of complex dimension $n$ are isomorphic as varieties if and only if they are weakly $T^n$-equivariantly unitary bordant (i.e. equivariantly unitary bordant after twisting one of the $T^n$-actions by an automorphism of $T^n$).
\end{theorem}

Compared with Masuda's equivariant cohomological rigidity theorem, our result does not require any additional condition on equivariant Chern classes. Weak equivariant unitary bordism alone suffices to determine the isomorphism type of a toric manifold. Furthermore, our theorem allows us to obtain complete invariants for toric manifolds up to isomorphism from known complete invariants of equivariant unitary bordism, as discussed in Sections \ref{subsection_chara} and \ref{subsection_Hirzebruch_surface}.

The same bordism rigidity phenomenon also occurs for quasitoric manifolds and small covers.
A matrix is said to satisfy the \emph{injectivity condition} if its columns are pairwise distinct. 

\begin{theorem} \label{main_theorem_quasitoric}
	Let $P_1, P_2\subset\mathbb R^n$ be $n$-dimensional oriented simple polytopes. Let $(P_1,\Lambda_1)$ and $(P_2,\Lambda_2)$ be combinatorial quasitoric pairs. Suppose that both $\Lambda_1$ and $\Lambda_2$ satisfy the injectivity condition. Then there exists a $T^n$-equivariant homeomorphism between $M(P_1, \Lambda_1)$ and $M(P_2, \Lambda_2)$ preserving omniorientations if and only if they are equivariantly unitary bordant.
\end{theorem}

\begin{theorem} \label{main_theorem_smallcover}
    Let $P_i\subset\mathbb R^n$ be an $n$-dimensional simple polytope, and $\lambda_i$ be a mod 2 characteristic function of $P_i$ for $i = 1, 2$. Suppose that both maps $\lambda_1$ and $\lambda_2$ are injective. Then $M(P_1, \lambda_1)$ and $M(P_2, \lambda_2)$ are $\mathbb{Z}_2^n$-equivariantly homeomorphic if and only if they are equivariantly unoriented bordant.
\end{theorem}

\begin{remark}
    Every omnioriented quasitoric manifold admits a $BT^\infty$-structure on its stable tangent bundle \cite{DavisJanus1991}.
    In \cite{Wiemeler2012}, Wiemeler proved that two $BT^\infty$-bordant omnioriented quasitoric manifolds are weakly equivariantly homeomorphic.

    Our main theorem \ref{main_theorem_quasitoric} differs from Wiemeler's result in several respects. 
    First, we work with equivariant unitary bordism, in which the torus action is retained throughout the bordism, whereas the $BT^\infty$-bordism considered by Wiemeler records the tangential $BT^\infty$-structure induced by the stable splitting and does not require the bordism itself to carry a torus action. 
    Second, under the injectivity condition, we show that equivariant unitary bordism is equivalent to the existence of an omniorientation-preserving $T^n$-equivariant homeomorphism. In particular, the conclusion is an actual $T^n$-equivariant homeomorphism, rather than merely a weakly equivariant one. Furthermore, Corollary \ref{cor_weak} extends this characterization to the weakly equivariant setting, establishing an equivalence between weakly equivariant unitary bordism and omniorientation-preserving weakly $T^n$-equivariant homeomorphism.
\end{remark}

For toric manifolds, the injectivity condition holds automatically, because distinct rays have distinct primitive generators. This condition is precisely what makes it possible for the equivariant bordism class to determine the isomorphism type.
Under the injectivity condition, distinct vertices have distinct tangent monomials. 
Thus equality of equivariant bordism classes induces a bijection between the fixed points preserving their tangent representations and fixed-point signs. 
The injectivity condition ensures that the combinatorial structure can be reconstructed from the fixed point data.
So the essential mechanism in our rigidity theorem is a local-to-global reconstruction: the injectivity condition turns local fixed-point representation data into global combinatorial data. 

The condition is automatic in several situations. In addition to toric manifolds, it also holds for the dual polytope of a simplicial 2-neighborly polytope (Proposition \ref{prop_2neighborly}).
We provide an example of equivariantly unitary bordant omnioriented quasitoric manifolds which, despite having distinct tangent monomials at distinct vertices, admit no omniorientation-preserving equivariant homeomorphism in the absence of the injectivity condition; see Example \ref{counterexample}.

A further consequence of our bordism rigidity theorems is that the known complete invariants of equivariant bordism become those of considerably finer geometric equivalence relations. 
Indeed, two unitary $T^k$-manifolds are equivariantly unitary bordant if and only if all integral equivariant cohomology Chern numbers agree \cite{LuWang2018}. Combining this with our rigidity theorem, we obtain the following conclusions (see Section \ref{subsection_chara}):
\begin{enumerate}
	\item \emph{The integral equivariant cohomology Chern numbers, up to the natural $\operatorname{GL}(n, \mathbb{Z})$-action on $H^*(BT^n; \mathbb{Z})$, form a complete invariant of the isomorphism type of toric manifolds.} 
	\item \emph{Among omnioriented quasitoric manifolds satisfying the injectivity condition, the integral equivariant cohomology Chern numbers completely determine omniorientation-preserving equivariant homeomorphism.}
\end{enumerate}
Similarly, tom Dieck's characterization by equivariant Stiefel--Whitney numbers \cite{tomDieck1971} and the real rigidity theorem show that the equivariant Stiefel--Whitney numbers form a complete system of invariants for the equivariant homeomorphism types of small covers with injective characteristic functions. 
Thus characteristic invariants which, in general, distinguish only equivariant bordism classes become complete geometric invariants within the rigid families considered here.

The paper is organized as follows. 
In Section~\ref{section_preliminary}, we review the necessary background on equivariant unitary bordism, toric varieties, quasitoric manifolds, and related topics. Section~\ref{section_main_thm} is mainly devoted to the proofs of the three rigidity theorems.
We discuss applications of these rigidity theorems to duals of 2-neighborly polytopes, to complete invariants and to the classification of Hirzebruch surfaces in Section~\ref{section_application}.

\section{Preliminaries} \label{section_preliminary}
In this section, we recall several basic notions, including equivariant unitary bordism, fans, toric varieties, and quasitoric manifolds, with particular emphasis on the combinatorial data associated with toric varieties and quasitoric manifolds and how these data determine the corresponding geometric objects.

\subsection{Equivariant unitary bordism}
Throughout this paper, $G$ is assumed to be a compact Lie group.

\begin{definition}
	A \emph{$G$-equivariant unitary manifold} is a smooth compact manifold $M$ equipped with a smooth $G$-action and an isomorphism of real $G$-vector bundles
	\begin{align*}
		TM\oplus \underline{\mathbb{R}}^{n}\longrightarrow \xi,
	\end{align*}
	where $\xi$ is a complex $G$-vector bundle and $\underline{\mathbb{R}}^n$ represents a real trivial bundle over $M$ with trivial $G$-action.
\end{definition}

Let $M$ be a $G$-equivariant unitary manifold, and $p\in M$ be an isolated fixed point. 
Then $T_pM$ is a complex $G$-representation, which induces an orientation of $T_pM$. On the other hand, the unitary structure on $M$ determines an orientation of $M$, and thus an orientation of $T_pM$. Therefore, $T_pM$ is equipped with two orientations.
The \emph{sign} of the isolated fixed point $p$, denoted by $\sigma(p)$, is defined by
\begin{align} \label{def_sign}
	\sigma(p) = 
	\begin{cases*}
		+1, & \text{if these two orientations agree},\\
		-1, & \text{if these two orientations disagree}.
	\end{cases*}
\end{align}

We now specialize to the case $G = T^k = (S^1)^k$. Let $M^{2n}$ be a $T^k$-equivariant unitary manifold with a finite set of fixed points.
Then the tangent representation at each fixed point $p\in M^{T^k}$ can be decomposed as
\begin{align} \label{decompo_tangent_rep}
	T_pM \cong \tau_1(p)\oplus \dots \oplus \tau_n(p),
\end{align}
where $\tau_1(p), \dots, \tau_n(p)$ are nontrivial irreducible complex $T^k$-representations. 
Toric manifolds and quasitoric manifolds, which will be reviewed later, form important classes of $T^k$-equivariant unitary manifolds with finitely many fixed points.

For a unitary $T^k$-manifold, its equivariant unitary bordism class is determined by the connected components of the fixed point set, together with their equivariant normal bundles; see \cite[Theorem 1]{HamrickOssa} and \cite[Theorem 1]{Hanke2005}. When the fixed points are isolated, this description simplifies, since the equivariant normal bundle at an isolated fixed point is a complex $T^k$-representation \cite[Proposition 3.5]{Darby2015}.

Denote by $\mathcal{Z}_*^{U, T^k}$ the equivariant unitary bordism ring consisting of the $T^k$-equivariant unitary bordism classes that admit representatives with isolated $T^k$-fixed points, and by $J_k$ the set of isomorphism classes of nontrivial irreducible complex $T^k$-representations. 
The fixed-point data described above induce a homomorphism of graded rings
\begin{equation} \label{fixed_point_homo}
	\begin{split}
		\varphi^U: \mathcal{Z}_*^{U, T^k} &\longrightarrow \mathbb{Z}[J_k] \\
		[M^{2n}] & \longmapsto \sum_{p\in M^{T^k}} \sigma(p)\prod_{i = 1}^n \tau_i(p).
	\end{split}
\end{equation}

\begin{lemma}
	The homomorphism $\varphi^U$ is injective.
\end{lemma}

Thus the equivariant unitary bordism class can be detected entirely from the local fixed-point data.

\subsection{Toric varieties}
We use the standard terminology and notation for polytopes, cones, fans, and normal fans; see, for example, \cite[Chapter 2]{BP_torictopo} for details. In particular, we will use the notions of rational, regular, and complete fans.


We now turn to toric varieties and their correspondence with rational fans.
The multiplicative group $(\mathbb{C}^\times)^n = (\mathbb{C}\setminus \{0\})^n$ contains the compact torus $T^n$ as a Lie subgroup.

\begin{definition}
	A \emph{toric variety} is a normal complex algebraic variety $X$ containing an algebraic torus $(\mathbb{C}^\times)^n$ as a Zariski open subset in such a way that the natural action of $(\mathbb{C}^\times)^n$ on itself extends to an algebraic action on $X$. A \emph{toric manifold} is a nonsingular complete (i.e., compact in the usual topology) toric variety.
\end{definition}

A fundamental result in toric geometry establishes a correspondence between $n$-dimensional toric varieties and rational fans in $n$-dimensional space. Specifically, nonsingular complete toric varieties correspond to complete regular fans. Thus, a toric manifold $X_\Sigma$ is determined by a complete regular fan $\Sigma$. 

\begin{lemma}[{\cite{Oda1988, MasudaSuh2008}}] \label{isom_variety_fan}
	For toric manifolds $X_{\Sigma_1}$ and $X_{\Sigma_2}$ of complex dimension $n$, the following are equivalent:
	\begin{enumerate}
		\item $X_{\Sigma_1}$ and $X_{\Sigma_2}$ are non-equivariantly isomorphic as varieties.
		\item $X_{\Sigma_1}$ and $X_{\Sigma_2}$ are weakly equivariantly isomorphic as varieties.
		\item There is $A\in \operatorname{GL}(n, \mathbb{Z})$ which maps cones of $\Sigma_1$ bijectively onto those of $\Sigma_2$.
	\end{enumerate}
\end{lemma}

This result allows us to determine whether two toric manifolds are isomorphic by comparing the combinatorial data of their associated fans. More precisely, the isomorphism problem can be reduced to checking whether their cones are related by an element of $\operatorname{GL}(n,\mathbb Z)$.
This serves as a first step in our study of the relationship between isomorphism and equivariant unitary bordism of toric manifolds.

For a toric manifold $X_\Sigma$, the $T^n$-equivariant unitary structure arises from its complex structure. Regarding $X_\Sigma$ as a $T^n$-equivariant unitary manifold, there is a one-to-one correspondence between its fixed points and $\Sigma(n)$, the set of maximal cones in $\Sigma$.
By definition, every fixed point has sign $+1$.
Moreover, the tangent representation at each fixed point can be computed directly from the corresponding maximal cone.

Let $\Sigma$ be a complete regular fan in $\mathbb{R}^n$. Set $\Sigma(k) = \{\sigma\in \Sigma\mid \dim\sigma = k\}$ for $0\leq k\leq n$. Suppose that $\Sigma(1) = \{l_1, \dots, l_m\}$, and let $\lambda_j$ denote the primitive generator of the ray $l_j$. 
Denote by $\Lambda_\Sigma$ the $n\times m$-matrix $(\lambda_1, \dots, \lambda_m)$.
Thus if $\sigma\in \Sigma(n)$ has rays $l_{j_1}, \dots, l_{j_n}$, where $1\leq j_1<\dots < j_n\leq m$, then $\sigma$ is generated by $\lambda_{j_1}, \dots, \lambda_{j_n}$. Since $\Sigma$ is regular, $\Lambda_\sigma = (\lambda_{j_1}, \dots, \lambda_{j_n})\in \operatorname{GL}(n,\mathbb Z)$.

\begin{lemma}[\cite{CLS2011}]
    Suppose that $p\in (X_\Sigma)^{T^n}$ corresponds to a maximal cone $\sigma\in \Sigma(n)$. 
	Let 
	\begin{align*}
		W_\sigma = (w_1(\sigma), \dots, w_n(\sigma)) \in \operatorname{GL}(n,\mathbb Z)
	\end{align*}
	be the matrix determined by
	\begin{align*}
		W_\sigma^\top \Lambda_\sigma = I_n.
	\end{align*}
	Then the tangent representation at $p$ decomposes as
	\begin{align*}
		T_p X_\Sigma\cong \tau_{w_1(\sigma)}\oplus \dots \oplus \tau_{w_n(\sigma)},
	\end{align*}
	where for $\alpha = (a_1, \dots, a_n)^\top\in \mathbb{Z}^n$, $\tau_\alpha: T^n\to \operatorname{GL}(1, \mathbb{C})$ denotes the one-dimensional complex $T^n$-representation defined by $\tau_\alpha(t_1, \dots, t_n) = t_1^{a_1}\dots t_n^{a_n}$.
\end{lemma}

Consequently, the image of $[X_\Sigma]$ under the fixed-point homomorphism $\varphi^U$ can be expressed explicitly in terms of the fan $\Sigma$ as,
\begin{align} \label{explicit_formula_toric}
	\varphi^U([X_\Sigma]) = \sum_{\sigma\in \Sigma(n)} \prod_{i=1}^n \tau_{w_i(\sigma)}.
\end{align}
Thus the $T^n$-equivariant unitary bordism class of a toric manifold is encoded by its fan $\Sigma$.
This provides another key ingredient in our study of the relationship between isomorphism and equivariant unitary bordism.

\subsection{Quasitoric manifolds}
There are several topological generalizations of toric varieties. In this section, we focus on quasitoric manifolds, introduced by Davis and Januszkiewicz~\cite{DavisJanus1991}. Omnioriented quasitoric manifolds admit natural $T^n$-equivariant unitary structures, and hence 
also provide a significant family of examples in $T^n$-equivariant unitary bordism. We refer to \cite[Chapter 7]{BP_torictopo} for further background on quasitoric manifolds.

Let $P$ be an $n$-dimensional simple polytope. A \emph{quasitoric manifold} over $P$ is a $2n$-dimensional closed smooth manifold $M$ with a locally standard $T^n$-action, whose orbit space is homeomorphic to the polytope $P$, as manifolds with corners. 
Every projective toric manifold $X_P$ is a quasitoric manifold over $P$.

Let $M$ be a quasitoric manifold over $P$, and let $\pi: M\to P$ be the orbit map. Denote by $\mathcal{F}(P) = \{F_1, \dots, F_m\}$ the set of facets of $P$. For each $1\le j\le m$, the preimage $M_j = \pi^{-1}(F_j)$ is a connected $T^n$-invariant submanifold of real codimension 2, called the \emph{characteristic submanifold} corresponding to $F_j$. 
Each $M_j$ is a connected component of the fixed point set of a circle subgroup $T_j\subset T^n$. The assignment
\begin{align} \label{charac_function}
	\lambda: \mathcal{F}(P)\longrightarrow \{\text{circle subgroup of } T^n\}, \quad F_j \longmapsto T_j,
\end{align}
is called the \emph{characteristic function} of $M$.
Equivalently, under the identification $\operatorname{Hom}(S^1, T^n)\cong \mathbb{Z}^n$, each $T_j$ is represented by a primitive vector $\lambda_j = (\lambda_{1j}, \dots, \lambda_{nj})^\top\in \mathbb{Z}^n$, uniquely up to sign, such that
\begin{align*}
	T_j = \{(z^{\lambda_{1j}}, \dots, z^{\lambda_{nj}})\in T^n\mid z\in S^1\}.
\end{align*}
The local standardness of the $T^n$-action implies the nonsingularity condition:
\begin{align} \label{nonsingular_condition}
	\text{if } F_{j_1}\cap \dots \cap F_{j_n}\ne \varnothing, \quad\det(\lambda_{j_1}, \dots, \lambda_{j_n}) = \pm 1.
\end{align}

The pair $(P, \lambda)$, called the \emph{characteristic pair}, records both the combinatorial structure of the orbit space and the isotropy data of the torus action. In particular, the characteristic pair determines
the quasitoric manifold up to equivariant homeomorphism.

Let $(P, \lambda)$ be a characteristic pair, and $p\in \partial P$. Denote by $F(p)$ the smallest face of $P$ containing $p$. Write $F(p) = F_{j_1}\cap \dots\cap F_{j_k}$. Define
\begin{align*}
	M(P, \lambda) = T^n\times P/\sim,
\end{align*}
where $(t, p)\sim (t', p)$ if and only if $t^{-1}t'\in T_{j_1}\times \dots \times T_{j_k}$. 
The left multiplication action of $T^n$ on itself descends to an action on $M(P, \lambda)$. Since $M(P, \lambda)$ is $T^n$-equivariantly homeomorphic to the quotient of the moment-angle manifold $\mathcal{Z}_P$ by a free $T^{m-n}$-action, $M(P, \lambda)$ has a canonical smooth structure.

\begin{lemma}[{\cite[Proposition 1.8]{DavisJanus1991}}]
	Let $M$ be a quasitoric manifold over $P$ with characteristic function $\lambda$. Then $M$ is $T^n$-equivariantly homeomorphic to the canonical model $M(P, \lambda)$, by a homeomorphism covering the identity map on $P$.
\end{lemma}

An \emph{omniorientation} of a quasitoric manifold $M$ is a choice of orientations on $M$ and on each characteristic submanifold $M_j$, $1\le j\le m$. 
A quasitoric manifold equipped with an omniorientation is called \emph{omnioriented}.
If $M$ is equipped with a $T^n$-invariant almost complex structure, then the complex orientations of $M$ and $M_j$ determine a canonical omniorientation. 
In particular, the complex structure of a toric manifold determines such an omniorientation.

An omniorientation of a quasitoric manifold canonically determines the signs of its characteristic vectors $\lambda_j$, as well as a $T^n$-equivariant unitary structure on the manifold.

Indeed, fixing an omniorientation, the stabilizer $T_j$ of $M_j$ is represented by a uniquely determined primitive vector $\lambda_j = (\lambda_{1j}, \dots, \lambda_{nj})^\top\in \mathbb{Z}^n$. The characteristic function \eqref{charac_function} can therefore be encoded by the integer $n \times m$-matrix
\begin{align*}
	\Lambda = (\lambda_1, \dots, \lambda_m) = (\lambda_{ij}),
\end{align*}
satisfying the nonsingularity condition \eqref{nonsingular_condition}. The matrix $\Lambda$ is called the \emph{characteristic matrix}.

We always equip $T^n$ with its canonical orientation. Then a choice of an orientation for $M$ is equivalent to a choice of an orientation for the polytope $P$.
The pair $(P, \Lambda)$, consisting of an oriented simple polytope $P^n$ with $m$ facets and an $n \times m$-characteristic matrix $\Lambda$, is called a \emph{combinatorial quasitoric pair}. We denote by $M(P, \Lambda)$ the corresponding omnioriented quasitoric manifold.

The omniorientation also determines a natural $T^n$-equivariant unitary structure on $M(P,\Lambda)$, as follows.

\begin{lemma}[{\cite[Theorem 6.6]{DavisJanus1991}}]
	There is an isomorphism of real $T^n$-vector bundles over $M(P, \Lambda)$,
	\begin{align*}
		TM(P, \Lambda) \oplus \underline{\mathbb{R}}^{2(m-n)} \cong \rho_1\oplus \dots \oplus \rho_m,
	\end{align*}
	where each $\rho_i$ is a $T^n$-equivariant complex line bundle over $M(P, \Lambda)$ and $\underline{\mathbb{R}}^{2(m-n)}$ denotes the trivial real $2(m-n)$-dimensional $T^n$-vector bundle over $M(P, \Lambda)$.
\end{lemma}

So the omnioriented quasitoric manifold $M(P, \Lambda)$ can be naturally regarded as a $T^n$-equivariant unitary manifold. The sign \eqref{def_sign} of each fixed point and the decomposition of the tangent representation \eqref{decompo_tangent_rep} at each fixed point can be computed directly from the combinatorial quasitoric pair. Denote by $o_P$ the orientation of $P$.

\begin{lemma}[{\cite[Lemma 7.3.19]{BP_torictopo}}] \label{sign_omni_quasi}
	Let $p = M_{j_1}\cap \dots\cap M_{j_n}$ be a fixed point. Then
	\begin{align} \label{sign_omniori_quasi}
		\sigma(p) = \det(\lambda_{j_1}, \dots, \lambda_{j_n})\operatorname{sign}_{o_P}(\mathbf{n}_{j_1}, \dots, \mathbf{n}_{j_n}),
	\end{align}
	where $\mathbf{n}_{j_1}, \dots, \mathbf{n}_{j_n}$ are inward-pointing normal vectors to the facets $F_{j_1}, \dots, F_{j_n}$, respectively, and
	\begin{align*}
		\operatorname{sign}_{o_P}(\mathbf{n}_{j_1}, \dots, \mathbf{n}_{j_n}) = 
		\begin{cases*}
			+1, & \text{the orientations of} $(\mathbf{n}_{j_1}, \dots, \mathbf{n}_{j_n})$ \text{and} $P$ \text{coincide},\\
			-1, & \text{the orientations of} $(\mathbf{n}_{j_1}, \dots, \mathbf{n}_{j_n})$ \text{and} $P$ \text{differ}.
		\end{cases*}
	\end{align*}
	If $v = F_{j_1}\cap \dots\cap F_{j_n}$ is the corresponding vertex of $p$, write $\sigma(v)\coloneq \sigma(p)$.
\end{lemma}

The formula \eqref{sign_omniori_quasi} is independent of the order of the facets, since a permutation changes the signs of the two determinants simultaneously.

\begin{lemma}[{\cite[Proposition 7.3.18]{BP_torictopo}}]
	Let $v = F_{j_1}\cap \dots\cap F_{j_n}$ with $1\le j_1 <\dots< j_n\le m$, and set
	\begin{align*}
		\Lambda_v = (\lambda_{j_1}, \dots, \lambda_{j_n}) \in \operatorname{GL}(n,\mathbb Z).
	\end{align*}
	Let 
	\begin{align*}
		W_v = (w_1(v), \dots, w_n(v)) \in \operatorname{GL}(n,\mathbb Z)
	\end{align*}
	be the matrix determined by
	\begin{align*}
		W_v^\top \Lambda_v = I_n.
	\end{align*}
	Then the tangent representation at $p = \pi^{-1}(v)$ decomposes as
	\begin{align*}
		T_p M(P, \Lambda)\cong \tau_{w_1(v)}\oplus \dots \oplus \tau_{w_n(v)}.
	\end{align*}
\end{lemma}

Combining the preceding two lemmas, $\varphi^U([M(P,\Lambda)])$ can be expressed explicitly in terms of $(P, \Lambda)$,
\begin{align} \label{explicit_formula}
	\varphi^U([M(P,\Lambda)]) = \sum_{v\in \mathcal{V}(P)} \sigma(v) \prod_{i=1}^n \tau_{w_i(v)},
\end{align}
where $W_{v} = (\Lambda_{v}^{-1})^{\top}$, the sign $\sigma(v)$ is determined by the formula \eqref{sign_omniori_quasi}, and $\mathcal{V}(P)$ is the set of vertices of $P$. 


\section{Proof of Main Theorems} \label{section_main_thm}
This section contains the main arguments of the paper. Our goal is to translate the $T^n$-equivariant unitary bordism problem for omnioriented quasitoric manifolds into a combinatorial one under a specific condition. We first establish a combinatorial characterization of the bordism relation and then apply it to prove our main theorems.

We begin with some notation and terminology.
Let $G$ be a compact Lie group, $\Phi$ be a Lie group automorphism of $G$, and $M$ be a smooth manifold equipped with a smooth $G$-action $\varphi$. Denote by $M_\Phi$ the smooth $G$-manifold with the same underlying smooth manifold as $M$, and with the $G$-action twisted by $\Phi$:
\begin{align*}
	G\times M \xrightarrow{\ \Phi\times\operatorname{id}_M\ } G\times M \stackrel{\varphi}{\longrightarrow}M.
\end{align*}

\begin{definition}
    Two $G$-equivariant unitary manifolds $M_1$ and $M_2$ are called \emph{weakly $G$-equivariantly unitary bordant} if there is an automorphism $\Phi$ of $G$ such that $M_1$ and $(M_2)_{\Phi}$ are equivariant unitary bordant.
\end{definition}

Recall that every Lie group automorphism of $T^n$ is uniquely determined by a matrix $A = (a_{ij})\in \operatorname{GL}(n, \mathbb{Z})$. Explicitly, the corresponding automorphism is
\begin{align*}
	\Phi_A: T^n\longrightarrow T^n, \quad \Phi_A(z_1, \dots, z_n) = \left( \prod_{j = 1}^n z_j^{a_{1j}}, \dots, \prod_{j = 1}^n z_j^{a_{nj}}\right).
\end{align*}

\subsection{Proof of Theorem \ref{main_theorem_toric}}
Suppose that $\Sigma$ is a complete regular fan in $\mathbb{R}^n$ with $\Lambda_\Sigma = (\lambda_1, \dots, \lambda_m)$, and $\sigma\in \Sigma(n)$. For convenience, we denote by $N(\sigma, \Sigma)$ the set of column vectors of $\Lambda_\sigma$, and by $T(\sigma, \Sigma)$ the set of column vectors of $W_\sigma$.
Since $W_\sigma^\top \Lambda_\sigma = I_n$, $N(\sigma, \Sigma)$ and $T(\sigma, \Sigma)$ form two bases of $\mathbb{Z}^n$ that are dual to each other.

\begin{proof}[Proof of Theorem \ref{main_theorem_toric}]
    Let $X_{\Sigma_1}$ and $X_{\Sigma_2}$ be two toric manifolds of complex dimension $n$ corresponding to complete regular fans $\Sigma_1$ and $\Sigma_2$ in $\mathbb{R}^n$, respectively.
	If $X_{\Sigma_1}$ and $X_{\Sigma_2}$ are isomorphic as varieties, then they are weakly equivariantly isomorphic by Lemma \ref{isom_variety_fan}. So they are weakly equivariantly biholomorphic as complex $T^n$-manifolds, which directly implies that they are weakly equivariantly unitary bordant.
	
	Conversely, suppose that $X_{\Sigma_1}$ and $(X_{\Sigma_2})_{\Phi_A}$ are $T^n$-equivariantly unitary bordant for some $A\in \operatorname{GL}(n, \mathbb{Z})$. First note that $(X_{\Sigma_2})_{\Phi_A}$ is still a $T^n$-equivariant unitary manifold, and $(X_{\Sigma_2})_{\Phi_A}^{T^n} = (X_{\Sigma_2})^{T^n}$. Since the twisted action does not change the underlying complex structure, all fixed points in the new action still have sign $+1$. Let $p\in (X_{\Sigma_2})^{T^n}$ be the fixed point corresponding to $\sigma\in \Sigma_2(n)$. Then 
    \begin{align*}
        T_p (X_{\Sigma_2})_{\Phi_A} \cong \tau_{A^\top w_1(\sigma)}\oplus \dots \oplus \tau_{A^\top w_n(\sigma)}.
    \end{align*}
    Since $X_{\Sigma_1}$ and $(X_{\Sigma_2})_{\Phi_A}$ are $T^n$-equivariantly unitary bordant, we have
    \begin{align} \label{eq_equal_poly}
        \sum_{\sigma\in \Sigma_1(n)} \prod_{i=1}^n \tau_{w_i(\sigma)} = \sum_{\sigma'\in \Sigma_2(n)} \prod_{i=1}^n \tau_{A^\top w_i(\sigma')}.
    \end{align}
    Since different maximal cones are generated by distinct sets of primitive generators, they determine distinct monomials and thus no cancellation occurs in either side. Hence the formula \eqref{eq_equal_poly} implies that there is a bijection $f: \Sigma_1(n)\to \Sigma_2(n)$ such that 
    \begin{align*}
        T(\sigma, \Sigma_1) = A^\top T(f(\sigma), \Sigma_2) = \{A^\top w_1(\sigma'), \dots, A^\top w_n(\sigma')\}.
    \end{align*}
    Then 
    \begin{align} \label{eq:match_N_toric}
        N(\sigma, \Sigma_1) = A^{-1} N(f(\sigma), \Sigma_2),
    \end{align}
    for all $\sigma\in \Sigma_1(n)$.
    Write $\Lambda_{\Sigma_i} = (\lambda^i_1, \dots, \lambda^i_{m_i})$. Since $\Sigma_1$ and $\Sigma_2$ are complete, the equation \eqref{eq:match_N_toric} implies that $\{A\lambda^1_1, \dots, A\lambda^1_{m_1}\} = \{\lambda^2_1, \dots, \lambda^2_{m_2}\}$.
    It follows that $m_1 = m_2 \eqcolon m$. 
    Regarding $A$ as an automorphism of $\mathbb{R}^n$, it defines a bijection $\Sigma_1(1)\to \Sigma_2(1)$, $\operatorname{cone}(\lambda^1_j)\mapsto \operatorname{cone}(A\lambda^1_j)$, still denoted this map by $A$.
    Also, $A$ induces a bijection $\Sigma_1(n)\to \Sigma_2(n)$, satisfying
    \begin{align*}
        A(\sigma) = A(\operatorname{cone}(N(\sigma, \Sigma_1))) = \operatorname{cone}(A(N(\sigma, \Sigma_1))) = \operatorname{cone}(N(f(\sigma), \Sigma_2)) = f(\sigma),
    \end{align*}
    for any $\sigma\in \Sigma(n)$.
    Therefore, $A$ maps cones of $\Sigma_1$ bijectively onto cones of $\Sigma_2$: for any cone $\sigma_1\in \Sigma_1$, choose a maximal cone $\sigma\in \Sigma_1(n)$ such that $\sigma_1\subseteq \sigma$. Since $A$ is a linear transformation, $A(\sigma_1)\subseteq A(\sigma)\in \Sigma_2(n)$. Hence, $A(\sigma_1)\in \Sigma_2(n)$. Therefore, $X_{\Sigma_1}$ and $X_{\Sigma_2}$ are isomorphic by Lemma \ref{isom_variety_fan}.
\end{proof}

\subsection{Proof of Theorems \ref{main_theorem_quasitoric} and \ref{main_theorem_smallcover}}
Suppose that $(P,\Lambda)$ is a combinatorial quasitoric pair and $v\in P$ is a vertex. Denote by
\begin{align*}
	N(v, \Lambda) = \{\lambda_{j_1}, \dots, \lambda_{j_n}\}
\end{align*}
the set of column vectors of $\Lambda_v$, and by
\begin{align*}
	T(v, \Lambda) = \{w_1(v), \dots, w_n(v)\}
\end{align*}
the set of column vectors of $W_v$.
Then $N(v,\Lambda)$ and $T(v,\Lambda)$ form two bases of $\mathbb{Z}^n$ that are dual to each other.

In the case of toric varieties, we use the fact that no terms cancel on either side of \eqref{eq_equal_poly}. Consequently, equality of the fixed-point polynomials induces a unimodular bijection between the fans.
This relies on the fact that different rays of a fan $\Sigma$ have different primitive generators, so the column vectors of the matrix $\Lambda_\Sigma$ are pairwise distinct. However, for a quasitoric manifold, distinct fixed points may have the same sign and tangent representation. We therefore impose an analogous restriction on the associated combinatorial quasitoric pairs.

\begin{definition}
	We say that a characteristic matrix $\Lambda$ satisfies the \emph{injectivity condition} if its column vectors are pairwise distinct.
\end{definition}

The injectivity condition implies that different vertices have different tangent representations.

\begin{lemma}\label{inj_imply_dif_tangent}
	Let $(P,\Lambda)$ be a combinatorial quasitoric pair such that $\Lambda$ satisfies the injectivity condition. If $v_1$ and $v_2$ are two distinct vertices of $P$, then
	\begin{align*}
		T(v_1,\Lambda)\neq T(v_2,\Lambda).
	\end{align*}
	Equivalently,
	\begin{align*}
		\prod_{i=1}^n\tau_{w_i(v_1)}
		\neq
		\prod_{i=1}^n\tau_{w_i(v_2)}
	\end{align*}
	as monomials in $\mathbb Z[J_n]$.
\end{lemma}

\begin{proof}
	Since $v_1\neq v_2$, they are incident to different sets of facets. Since the column vectors of $\Lambda$ are pairwise distinct, it follows that $N(v_1,\Lambda)\neq N(v_2,\Lambda)$. Since dual bases are uniquely determined, it follows that $T(v_1,\Lambda)\neq T(v_2,\Lambda)$. Finally, for $\alpha,\beta\in\mathbb Z^n$, the one-dimensional complex $T^n$-representations $\tau_\alpha$ and $\tau_\beta$ are isomorphic if and only if $\alpha=\beta$. Therefore,
	\begin{align*}
		\prod_{i=1}^n\tau_{w_i(v_1)} \ne \prod_{i=1}^n\tau_{w_i(v_2)}.
	\end{align*}
\end{proof}



Denote by $\mathcal{P}(P)$ the face poset of $P$. Let $P_1$ and $P_2$ be oriented simple polytopes. A poset isomorphism $h: \mathcal{P}(P_1)\to \mathcal{P}(P_2)$ is said to be \emph{orientation-preserving} if for each vertex $v\in P_1$, say $v = F_{1, j_1}\cap\cdots\cap F_{1, j_n}$, $h(F_{1,j_r})=F_{2,k_r}$,
\begin{align*}
	\operatorname{sign}_{o_{P_1}}(\mathbf{n}_{1, j_1}, \dots, \mathbf{n}_{1, j_n}) = \operatorname{sign}_{o_{P_2}}(\mathbf{n}_{2, k_1}, \dots, \mathbf{n}_{2, k_n}).
\end{align*}

We are now ready to give a combinatorial criterion of the $T^n$-equivariant unitary bordism relation between omnioriented quasitoric manifolds satisfying the injectivity condition.

\begin{theorem} \label{combinatorial_equiv_bordism}
	Let $P_1, P_2\subset\mathbb R^n$ be $n$-dimensional oriented simple polytopes. Let $(P_1,\Lambda_1)$ and $(P_2,\Lambda_2)$ be combinatorial quasitoric pairs. Suppose that both $\Lambda_1$ and $\Lambda_2$ satisfy the injectivity condition. Then the following statements are equivalent:
	\begin{enumerate}
		\item $[M(P_1,\Lambda_1)] = [M(P_2,\Lambda_2)]\in \mathcal Z_{2n}^{U,T^n}$.
		
		\item There exists an orientation-preserving poset isomorphism $h: \mathcal{P}(P_1)\to \mathcal{P}(P_2)$, such that
		\begin{align*}
			\Lambda_1=h^*\Lambda_2.
		\end{align*}
	\end{enumerate}
	More precisely, write $\mathcal F(P_i) = \{F_{i,1},\dots,F_{i,m_i}\}$, $\Lambda_i = (\lambda^i_1,\dots,\lambda^i_{m_i})$.
	If $h$ is an isomorphism, then necessarily $m_1=m_2\eqcolon m$, and it induces a permutation $\theta\in \mathcal{S}_m$ determined by
	\begin{align*}
		h(F_{1,j}) = F_{2,\theta(j)}, \quad 1\leq j\leq m.
	\end{align*}
	We define
	\begin{align*}
		h^*\Lambda_2 \coloneq (\lambda^2_{\theta(1)},\dots,\lambda^2_{\theta(m)}).
	\end{align*}
\end{theorem}

\begin{proof}
	Let $\pi_i: M(P_i,\Lambda_i)\to P_i$, $i=1,2$, be the orbit maps. By injectivity of the fixed-point homomorphism $\varphi^U$, we have
	\begin{align*}
		[M(P_1,\Lambda_1)] = [M(P_2,\Lambda_2)] \iff \varphi^U([M(P_1,\Lambda_1)]) = \varphi^U([M(P_2,\Lambda_2)]).
	\end{align*}
	By \eqref{explicit_formula}, the latter equality is equivalent to
	\begin{align}\label{equ1}
		\sum_{v\in \mathcal{V}(P_1)} \sigma(v) \prod_{k=1}^n\tau_{w_k(v)} = \sum_{u\in \mathcal{V}(P_2)} \sigma(u) \prod_{k=1}^n\tau_{w_k(u)}.
	\end{align}
	By Lemma~\ref{inj_imply_dif_tangent}, the monomials occurring on each side of \eqref{equ1} are pairwise distinct. Hence no cancellation occurs within either side. Therefore, \eqref{equ1} holds if and only if there exists a bijection $f: \mathcal{V}(P_1)\to \mathcal{V}(P_2)$, such that for every $v\in \mathcal{V}(P_1)$,
	\begin{align}\label{eq:match_sign}
		\sigma(v) = \sigma(f(v))
	\end{align}
	and
	\begin{align}\label{eq:match_T}
		T(v,\Lambda_1) = T(f(v),\Lambda_2).
	\end{align}
	
	We first prove $(1)\Rightarrow(2)$. Assume that $[M(P_1,\Lambda_1)] = [M(P_2,\Lambda_2)]$. Then there exists a bijection $f: \mathcal{V}(P_1)\to \mathcal{V}(P_2)$ satisfying \eqref{eq:match_sign} and \eqref{eq:match_T}. So
	\begin{align} \label{eq:match_N}
		N(v,\Lambda_1) = N(f(v),\Lambda_2),
	\end{align}
	for every $v\in \mathcal{V}(P_1)$.
	Hence, every column of $\Lambda_1$ belongs to $N(v,\Lambda_1)$ for some $v\in \mathcal{V}(P_1)$. By \eqref{eq:match_N}, it is also a column of $\Lambda_2$. Applying the same argument to $f^{-1}$ yields $\{\lambda^1_1,\dots,\lambda^1_{m_1}\} = \{\lambda^2_1,\dots,\lambda^2_{m_2}\}$.
	Since the column vectors of both characteristic matrices are pairwise distinct, it follows that $m_1=m_2\eqcolon m$ and there exists a unique permutation $\theta\in \mathcal{S}_m$ such that
	\begin{align}\label{eq:column_correspondence}
		\lambda^1_j = \lambda^2_{\theta(j)} \quad \text{ for }1\leq j\leq m.
	\end{align}
	For every $v\in \mathcal{V}(P_1)$ and every facet $F_{1,j}$, we have
	\begin{align*}
		v\in F_{1,j} \iff \lambda^1_j\in N(v,\Lambda_1) \iff \lambda^2_{\theta(j)} \in N(f(v),\Lambda_2) \iff f(v)\in F_{2,\theta(j)}.
	\end{align*}
	Thus the bijections $v\mapsto f(v)$ on vertices and $F_{1,j}\mapsto F_{2,\theta(j)}$ on facets preserve all vertex-facet incidences, and these bijections determine an isomorphism $h: \mathcal{P}(P_1)\to \mathcal{P}(P_2)$ such that $h(F_{1, j}) = F_{2, \theta(j)}$. 
	By \eqref{eq:column_correspondence}, $\Lambda_1=h^*\Lambda_2$.
	
	It remains to prove that $h$ preserves the orientation. Choose any vertex $v = F_{1,j _1}\cap\cdots\cap F_{1, j_n}$. Then $f(v) = F_{2, \theta(j_1)} \cap\cdots\cap F_{2, \theta(j_n)}$. Let $\mathbf{n}_{i, j}$ be an inward-pointing normal vector to the facet $F_{i, j}$. By the fixed-point sign formula \eqref{sign_omniori_quasi},
	\begin{align*}
		\sigma(v) &= \det(\lambda^1_{j_1},\dots,\lambda^1_{j_n})\operatorname{sign}_{o_{P_1}} (\mathbf{n}_{1,j_1},\dots,\mathbf{n}_{1,j_n}), \\
		\sigma(f(v)) &= \det(\lambda^2_{\theta(j_1)},
		\dots, \lambda^2_{\theta(j_n)})\operatorname{sign}_{o_{P_2}}(\mathbf{n}_{2,\theta(j_1)}, \dots, \mathbf{n}_{2,\theta(j_n)}).
	\end{align*}
	By \eqref{eq:column_correspondence}, $(\lambda^1_{j_1},\dots,\lambda^1_{j_n}) = (\lambda^2_{\theta(j_1)}, \dots, \lambda^2_{\theta(j_n)})$, while \eqref{eq:match_sign} gives $\sigma(v) = \sigma(f(v))$. Therefore, 
	\begin{align*}
		\operatorname{sign}_{o_{P_1}} (\mathbf{n}_{1,j_1},\dots,\mathbf{n}_{1,j_n}) = \operatorname{sign}_{o_{P_2}}(\mathbf{n}_{2,\theta(j_1)}, \dots, \mathbf{n}_{2,\theta(j_n)}),
	\end{align*}
	which indicates that $h$ is orientation-preserving. Hence $(1)\Rightarrow(2)$.
	
	Conversely, suppose that there exists an orientation-preserving isomorphism $h: \mathcal{P}(P_1)\to  \mathcal{P}(P_2)$, such that $\Lambda_1 = h^*\Lambda_2$.
	Let $f: \mathcal{V}(P_1)\to \mathcal{V}(P_2)$ be the induced bijection of vertices. If $v = F_{1, j_1}\cap\cdots\cap F_{1, j_n}$, then $f(v) = F_{2,\theta(j_1)} \cap\cdots\cap F_{2,\theta(j_n)}$. Since $\lambda^1_j = \lambda^2_{\theta(j)}$ for every $j$, we have $N(v,\Lambda_1) = N(f(v),\Lambda_2)$.
	Equivalently, after taking the dual bases, $T(v,\Lambda_1) = T(f(v),\Lambda_2)$.
	Moreover, since $h$ is orientation-preserving, then $\operatorname{sign}_{o_{P_1}} (\mathbf{n}_{1,j_1},\dots, \mathbf{n}_{1,j_n}) = \operatorname{sign}_{o_{P_2}}(\mathbf{n}_{2,\theta(j_1)}, \dots, \mathbf{n}_{2,\theta(j_n)})$.
	Together with $(\lambda^1_{j_1},\dots,\lambda^1_{j_n}) = (\lambda^2_{\theta(j_1)}, \dots, \lambda^2_{\theta(j_n)})$,
	the fixed-point sign formula \eqref{sign_omniori_quasi} yields $\sigma(v) = \sigma(f(v))$.
	Therefore, $[M(P_1,\Lambda_1)] = [M(P_2,\Lambda_2)]\in\mathcal Z_*^{U,T^n}$. Thus $(2)\Rightarrow(1)$.
\end{proof}

This theorem gives a combinatorial characterization of when two omnioriented quasitoric manifolds are equivariant unitary bordant.
We can see from the proof that the property (2) always implies the property (1), while the injectivity condition is used only for the converse implication. 
The injectivity condition guarantees that distinct vertices give rise to distinct tangent-representation monomials, so that the local fixed-point terms can be identified individually from the fixed-point expression of the equivariant unitary bordism class. This makes it possible to recover the correspondence between vertices, and ultimately the combinatorial data of the omnioriented quasitoric manifolds.

We point out that the result fails without the injectivity condition, even if distinct fixed points have distinct tangent representations.

\begin{example} \label{counterexample}
	Consider a counterclockwise-oriented nonagon $P\subset \mathbb{R}^2$ and the combinatorial quasitoric pairs $(P, \Lambda_i)$ shown in Figure \ref{figure_counterexam} for $i = 1, 2$. 	
    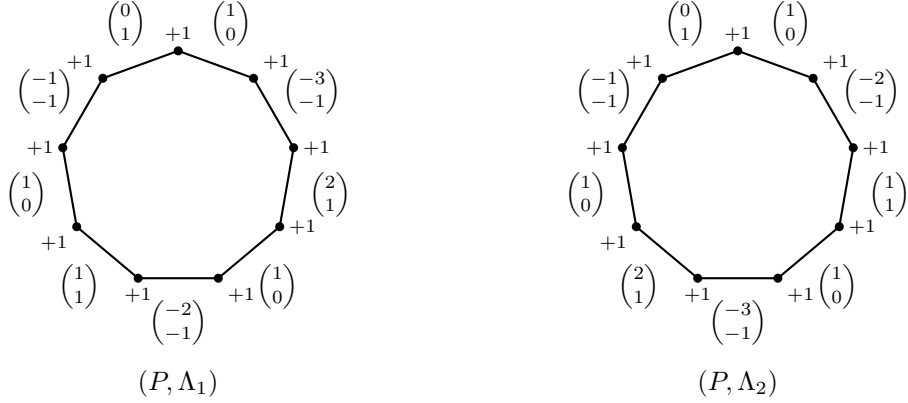
\begin{figure}[t]
		\centering
		
		\begin{tikzpicture}[
			vertex/.style={circle, fill=black, inner sep=1.2pt},
			vlabel/.style={font=\scriptsize},
			elabel/.style={font=\scriptsize, inner sep=1.2pt}
			]
			
			\begin{scope}[xshift=0cm]
				\coordinate (C1) at (0,0);
				
				\foreach \i in {1,...,9}{
					\coordinate (A\i) at ({90-(\i-1)*40}:1.55);
				}
				
				\foreach \i [evaluate=\i as \j using {int(mod(\i,9)+1)}] in {1,...,9}{
					\draw[thick] (A\i) -- (A\j);
				}
				
				\foreach \i in {1,...,9}{
					\node[vertex] at (A\i) {};
				}
				
				\node[vlabel, above]       at (A1) {$+1$};
				\node[vlabel, above right] at (A2) {$+1$};
				\node[vlabel, right]       at (A3) {$+1$};
				\node[vlabel, right]       at (A4) {$+1$};
				\node[vlabel, below right] at (A5) {$+1$};
				\node[vlabel, below]       at (A6) {$+1$};
				\node[vlabel, below left]  at (A7) {$+1$};
				\node[vlabel, left]        at (A8) {$+1$};
				\node[vlabel, above left]  at (A9) {$+1$};
				
				\coordinate (Am12) at ($(A1)!0.5!(A2)$);
				\coordinate (Am23) at ($(A2)!0.5!(A3)$);
				\coordinate (Am34) at ($(A3)!0.5!(A4)$);
				\coordinate (Am45) at ($(A4)!0.5!(A5)$);
				\coordinate (Am56) at ($(A5)!0.5!(A6)$);
				\coordinate (Am67) at ($(A6)!0.5!(A7)$);
				\coordinate (Am78) at ($(A7)!0.5!(A8)$);
				\coordinate (Am89) at ($(A8)!0.5!(A9)$);
				\coordinate (Am91) at ($(A9)!0.5!(A1)$);
				
				\node[elabel] at ($(C1)!1.40!(Am12)$) {$\begin{pmatrix}1\\0\end{pmatrix}$};
				\node[elabel] at ($(C1)!1.40!(Am23)$) {$\begin{pmatrix}-3\\-1\end{pmatrix}$};
				\node[elabel] at ($(C1)!1.40!(Am34)$) {$\begin{pmatrix}2\\1\end{pmatrix}$};
				\node[elabel] at ($(C1)!1.40!(Am45)$) {$\begin{pmatrix}1\\0\end{pmatrix}$};
				\node[elabel] at ($(C1)!1.40!(Am56)$) {$\begin{pmatrix}-2\\-1\end{pmatrix}$};
				\node[elabel] at ($(C1)!1.40!(Am67)$) {$\begin{pmatrix}1\\1\end{pmatrix}$};
				\node[elabel] at ($(C1)!1.40!(Am78)$) {$\begin{pmatrix}1\\0\end{pmatrix}$};
				\node[elabel] at ($(C1)!1.40!(Am89)$) {$\begin{pmatrix}-1\\-1\end{pmatrix}$};
				\node[elabel] at ($(C1)!1.40!(Am91)$) {$\begin{pmatrix}0\\1\end{pmatrix}$};
				
				\node[font=\normalsize] at (0,-2.85) {$(P, \Lambda_1)$};
			\end{scope}
			
			\begin{scope}[xshift=7.4cm]
				\coordinate (C2) at (0,0);
				
				\foreach \i in {1,...,9}{
					\coordinate (B\i) at ({90-(\i-1)*40}:1.55);
				}
				
				\foreach \i [evaluate=\i as \j using {int(mod(\i,9)+1)}] in {1,...,9}{
					\draw[thick] (B\i) -- (B\j);
				}
				
				\foreach \i in {1,...,9}{
					\node[vertex] at (B\i) {};
				}
				
				\node[vlabel, above]       at (B1) {$+1$};
				\node[vlabel, above right] at (B2) {$+1$};
				\node[vlabel, right]       at (B3) {$+1$};
				\node[vlabel, right]       at (B4) {$+1$};
				\node[vlabel, below right] at (B5) {$+1$};
				\node[vlabel, below]       at (B6) {$+1$};
				\node[vlabel, below left]  at (B7) {$+1$};
				\node[vlabel, left]        at (B8) {$+1$};
				\node[vlabel, above left]  at (B9) {$+1$};
				
				\coordinate (Bm12) at ($(B1)!0.5!(B2)$);
				\coordinate (Bm23) at ($(B2)!0.5!(B3)$);
				\coordinate (Bm34) at ($(B3)!0.5!(B4)$);
				\coordinate (Bm45) at ($(B4)!0.5!(B5)$);
				\coordinate (Bm56) at ($(B5)!0.5!(B6)$);
				\coordinate (Bm67) at ($(B6)!0.5!(B7)$);
				\coordinate (Bm78) at ($(B7)!0.5!(B8)$);
				\coordinate (Bm89) at ($(B8)!0.5!(B9)$);
				\coordinate (Bm91) at ($(B9)!0.5!(B1)$);
				
				\node[elabel] at ($(C2)!1.40!(Bm12)$) {$\begin{pmatrix}1\\0\end{pmatrix}$};
				\node[elabel] at ($(C2)!1.40!(Bm23)$) {$\begin{pmatrix}-2\\-1\end{pmatrix}$};
				\node[elabel] at ($(C2)!1.40!(Bm34)$) {$\begin{pmatrix}1\\1\end{pmatrix}$};
				\node[elabel] at ($(C2)!1.40!(Bm45)$) {$\begin{pmatrix}1\\0\end{pmatrix}$};
				\node[elabel] at ($(C2)!1.40!(Bm56)$) {$\begin{pmatrix}-3\\-1\end{pmatrix}$};
				\node[elabel] at ($(C2)!1.40!(Bm67)$) {$\begin{pmatrix}2\\1\end{pmatrix}$};
				\node[elabel] at ($(C2)!1.40!(Bm78)$) {$\begin{pmatrix}1\\0\end{pmatrix}$};
				\node[elabel] at ($(C2)!1.40!(Bm89)$) {$\begin{pmatrix}-1\\-1\end{pmatrix}$};
				\node[elabel] at ($(C2)!1.40!(Bm91)$) {$\begin{pmatrix}0\\1\end{pmatrix}$};
				
				\node[font=\normalsize] at (0,-2.85) {$(P, \Lambda_2)$};
			\end{scope}
		\end{tikzpicture}
		\caption{A counterexample of Theorem \ref{combinatorial_equiv_bordism} without the injectivity condition}
		\label{figure_counterexam}
	\end{figure}
    The vectors adjacent to the edges are the corresponding characteristic vectors of the edges, and the sign at each fixed point is $+1$. 
	It follows directly from the fixed-point data that $M(P, \Lambda_1)$ and $M(P, \Lambda_2)$ represent the same nonzero bordism class in $\mathcal{Z}_4^{U, T^2}$. However, neither $\Lambda_1$ nor $\Lambda_2$ satisfies the injectivity condition, and there is no orientation-preserving combinatorial automorphism $h$ of $P$ such that $\Lambda_1 = h^*\Lambda_2$.
\end{example}

Recall that for an omnioriented quasitoric manifold $M(P, \Lambda)$ and a Lie group automorphism $\Phi_A$ of $T^n$, $M(P, \Lambda)_{\Phi_A}$ is still an omnioriented quasitoric manifold whose corresponding combinatorial quasitoric pair is $(P_A, A^{-1}\Lambda)$, where $P_A$ is the same polytope as $P$ with orientation $\operatorname{sign}(\det A) o_P$. Then the following conclusion is a direct corollary.

\begin{corollary} \label{cor_weak_equiv_unitary_bor}
	Keep the notation and conventions of Theorem \ref{combinatorial_equiv_bordism}. Fix $A\in\operatorname{GL}(n,\mathbb Z)$. Then the following statements are equivalent:
	\begin{enumerate}
		\item $[M(P_1,\Lambda_1)] = [M(P_2,\Lambda_2)_{\Phi_A}]\in \mathcal Z_*^{U,T^n}$. 
		
		\item There exists an isomorphism $h: \mathcal{P}(P_1)\to \mathcal{P}(P_2)$, such that
		\begin{align*}
			A \Lambda_1=h^*\Lambda_2,
		\end{align*}
		where $h$ is orientation-preserving if $\det A>0$, and orientation-reversing if $\det A<0$.
	\end{enumerate}
\end{corollary}

\begin{proof}[Proof of Theorem \ref{main_theorem_quasitoric}]
	By \cite[Proposition 7.3.11]{BP_torictopo}, there exists a $T^n$-equivariant homeomorphism between $M(P_1, \Lambda_1)$ and $M(P_2, \Lambda_2)$ preserving omniorientations if and only if there exists a homeomorphism $\widetilde{h}: P_1\to P_2$ preserving face relations and orientations such that $\Lambda_1 = \widetilde{h}^*\Lambda_2$. 
	So it suffices to prove that they are also equivalent to the property (2) in Theorem \ref{combinatorial_equiv_bordism}. First, such a homeomorphism $\widetilde{h}: P_1\to P_2$ induces an isomorphism $h: \mathcal{P}(P_1)\to \mathcal{P}(P_2)$ satisfying all conditions in the property (2) in Theorem \ref{combinatorial_equiv_bordism}. 
	
    Conversely, suppose that there is an orientation-preserving isomorphism $h: \mathcal{P}(P_1)\to \mathcal{P}(P_2)$, such that $\Lambda_1=h^*\Lambda_2$. Then by \cite[Corollary 5.2]{Wiemeler2013}, there is a diffeomorphism $\widetilde{h}: P_1\to P_2$ (with the natural smooth structure coming from the embedding $P_i\hookrightarrow \mathbb{R}^n$) such that $\widetilde{h}(F) = h(F)$ for each face $F$ of $P_1$. This diffeomorphism clearly preserves orientations and face relations.
\end{proof}

Example \ref{counterexample} also gives a counterexample of Theorem \ref{main_theorem_quasitoric} without the injectivity condition. It shows that $M(P, \Lambda_1)$ and $M(P, \Lambda_2)$ are equivariantly unitary bordant, but there is no equivariant homeomorphism preserving omniorientations.

\begin{corollary} \label{cor_weak}
    Use the same notation and conventions of Theorem \ref{combinatorial_equiv_bordism}. Then there exists a weakly $T^n$-equivariant homeomorphism between $M(P_1, \Lambda_1)$ and $M(P_2, \Lambda_2)$ preserving omniorientations if and only if they are weakly equivariantly unitary bordant.
\end{corollary}

Davis and Januszkiewicz~\cite{DavisJanus1991} also introduced a real analogue of quasitoric manifolds, called \emph{small covers}. A small cover is a closed smooth $n$-dimensional manifold $M$ equipped with a locally standard action of $\mathbb Z_2^n$, whose orbit space is a simple $n$-polytope $P$. As in the quasitoric case, such a manifold is determined by its characteristic data. More precisely, there is a mod 2 characteristic function on $P$
\begin{align*}
	\lambda:\mathcal F(P)\longrightarrow \mathbb Z_2^n
\end{align*}
such that, when $F_{j_1} \cap \cdots \cap F_{j_n}$ is a vertex of $P$, the vectors $\lambda(F_{j_1}),\ldots,\lambda(F_{j_n})$ form a basis of $\mathbb Z_2^n$. We denote the corresponding small cover by $M(P,\lambda)$.

The equivariant bordism rigidity for quasitoric manifolds admits a real analogue for small covers. In fact, the proof is simpler in this setting. Because omniorientations do not need to be considered, the characteristic vectors have no sign ambiguity over $\mathbb Z_2$, and no fixed-point signs occur in the fixed-point polynomials. 
Consequently, under the same injectivity assumption on the characteristic functions, the equivariant unoriented bordism class already determines the $\mathbb Z_2^n$-equivariant homeomorphism type, and Theorem \ref{main_theorem_smallcover} follows.

\section{Applications of Main Theorems} \label{section_application}
In this section, we give some applications of Theorems \ref{main_theorem_toric}, \ref{main_theorem_quasitoric}, and \ref{main_theorem_smallcover}. We first apply these results to classes of manifolds for which the injectivity condition holds automatically. We then discuss complete invariants arising from equivariant bordism theory and apply them to the classification of Hirzebruch surfaces under various equivalence relations.

\subsection{Products of simplices}
Products of simplices form an important class of orbit polytopes in toric topology. In equivariant bordism theory, small covers and omnioriented quasitoric manifolds over such products generate the equivariant unoriented and unitary bordism groups of $2$-torus manifolds and unitary torus manifolds, respectively~\cite{LuTan2014, CLT2025}. Motivated by these results, we investigate quasitoric manifolds and small covers over products of simplices from the viewpoint of equivariant bordism rigidity.

A polytope is called \emph{$k$-neighborly} if any set of its $k$ or fewer vertices spans a proper face. Every simplex $\triangle^n$ is $n$-neighborly. Denote by $P^*$ the dual polytope of a given polytope $P$.

If $P_1^*$ is $k_1$-neighborly and $P_2^*$ is $k_2$-neighborly, then $(P_1\times P_2)^*$ is $\min(k_1, k_2)$-neighborly.
If $P^*$ is 2-neighborly, any 2 vertices are adjacent, and then the intersection of any two facets of $P$ is nonempty. In other words, if $P^*$ is a simplicial 2-neighborly polytope, every characteristic function on $P$ is injective. In particular, it holds for $P = \triangle^{n_1}\times \dots\times \triangle^{n_l}$ with $n_i\ge 2$.

\begin{proposition} \label{prop_2neighborly}
	Let $P^*$ be a simplicial 2-neighborly polytope. 
	\begin{enumerate}
		\item Suppose that $\Lambda_1$ and $\Lambda_2$ are characteristic matrices on $P$. Then $M(P, \Lambda_1)$ and $M(P, \Lambda_2)$ are equivariantly unitary bordant if and only if there exists an omniorientation-preserving equivariant homeomorphism between them.
		
		\item Suppose that $\lambda_1$ and $\lambda_2$ are mod 2 characteristic functions on $P$. Then $M(P, \lambda_1)$ and $M(P, \lambda_2)$ are equivariantly bordant if and only if they are equivariantly homeomorphic. 
	\end{enumerate}
\end{proposition}

This result is a direct corollary of Theorems \ref{main_theorem_quasitoric} and \ref{main_theorem_smallcover}, and can be applied to the case when $P$ is a product of simplices $\triangle^{n_1}\times \dots\times \triangle^{n_l}$ with $n_i\ge 2$. Next we consider the case when the product contains one-dimensional simplices.

\begin{lemma} \label{lemma_product_with_1dim}
	Suppose that $P = \triangle^{n_1}\times \dots\times \triangle^{n_l}$ with $n_i\ge 2$, $n = n_1 + \dots + n_l$, and $I_j$ is a copy of the 1-simplex, $1\le j\le m$. A characteristic matrix $\Lambda$ of $P\times \prod_{j = 1}^m I_j$ satisfies the injectivity condition if and only if $M(P\times \prod_{i = j}^m I_j, \Lambda)$ does not equivariantly unitarily bound.
	Similarly, a mod 2 characteristic function $\lambda$ of $P\times \prod_{j = 1}^m I_j$ is injective if and only if $M(P\times \prod_{i = j}^m I_j, \lambda)$ does not equivariantly bound.
\end{lemma}

\begin{proof}
	We only prove the unitary case, as it is parallel to the unoriented case.
	Let $F_1, \dots, F_{n+l}$ be all facets of $P$, and $v_{j, 0}$ and $v_{j, 1}$ be two vertices of $I_j$. So all facets of $P\times \prod_{j = 1}^m I_j$ are 
	\begin{align*}
		H_i &= F_i\times \prod_{j = 1}^m I_j, \\
		H_{j, \varepsilon} &= P\times I_1\times \dots\times I_{j-1}\times v_{j, \varepsilon}\times I_{j+1}\times\dots\times I_m,
	\end{align*}
	for $1\le i\le n+l$, $1\le j\le m$ and $\varepsilon = 0, 1$. 
	For any $\varepsilon_j\in \{0, 1\}$ and $1\le j\le m$, any two of the following facets
	\begin{align*}
		H_1, \dots, H_{n+l}, H_{1, \varepsilon_1}, \dots, H_{m, \varepsilon_m}
	\end{align*}
	have a nonempty intersection. 
    Denote by $\Lambda(F)$ the column of $\Lambda$ corresponding to the facet $F$.
    Thus $\Lambda(H_1), \dots, \Lambda(H_{n+l}), \Lambda(H_{1, \varepsilon_1}), \dots, \Lambda(H_{m, \varepsilon_m})$ are mutually distinct. 
	Therefore, $\Lambda$ does not satisfy the injectivity condition if and only if there exists an integer $1\le j_0\le m$ such that $\Lambda(H_{j_0, 0}) = \Lambda(H_{j_0, 1})$. 
	Hence, two fixed points
	\begin{align*}
		H_{i_1}\cap \dots\cap H_{i_n}\cap \cap_{j\ne j_0} H_{j, \varepsilon_j} \cap H_{j_0, 0}, \quad \text{ and } \quad H_{i_1}\cap \dots\cap H_{i_n}\cap \cap_{j\ne j_0} H_{j, \varepsilon_j} \cap H_{j_0, 1}
	\end{align*}
	have the same tangent representation and opposite signs,
	for any $H_{i_1}\cap \dots\cap H_{i_n}\ne \varnothing$ and $\varepsilon_j\in \{0, 1\}, j\ne j_0$. Since 
	\begin{align*}
		\mathcal{V}\left(P\times \prod_{i = j}^m I_j\right) = \{H_{i_1}\cap \dots\cap H_{i_n}\cap \cap_{j = 1}^m H_{j, \varepsilon_j}\mid H_{i_1}\cap \dots\cap H_{i_n}\ne \varnothing, \varepsilon_j\in \{0, 1\}\},
	\end{align*}
	the monomials in the polynomial $\varphi^U([M(P\times \prod_{i = 1}^m I_i, \Lambda)])$ cancel pairwise, and so $M(P\times \prod_{i = 1}^m I_i, \Lambda)$ equivariantly unitarily bounds.
	The converse is clear.
\end{proof}

Thus, in this setting, whether the corresponding quasitoric manifold or small cover equivariantly bounds is completely determined by the injectivity of its characteristic data.
In view of Theorems \ref{main_theorem_quasitoric}, \ref{main_theorem_smallcover} and Lemma \ref{lemma_product_with_1dim}, Proposition \ref{prop_2neighborly} can be reformulated for products of simplices as follows.

\begin{proposition}
    Let $P = \triangle^{n_1}\times \dots\times \triangle^{n_l}$ be a product of oriented simplices. Let $M_1$ and $M_2$ be omnioriented quasitoric manifolds (resp. small covers) over $P$. Assume that they do not $T^n$-equivariantly unitarily bound (resp. $\mathbb{Z}_2^n$-equivariantly unoriented bound). Then they are equivariantly unitary bordant (resp. equivariantly unoriented bordant) if and only if there exists an omniorientation-preserving equivariant homeomorphism between them (resp. they are equivariantly homeomorphic).
\end{proposition}

In particular, nontrivial equivariant bordism classes represented by quasitoric manifolds or small covers over products of simplices exhibit a rigidity phenomenon: two such representatives are bordant only when they are already equivariantly homeomorphic. Thus, in the nonbounding case, equivariant bordism classification and equivariant homeomorphism classification coincide.

\begin{remark}
    The nonbounding assumption in the proposition is essential. Indeed, every small cover over a cube equivariantly bounds \cite[Corollary 6.4]{LuTan2014}. Consequently, any two small covers over cubes represent the same equivariant bordism class, although they need not be equivariantly homeomorphic.
\end{remark}

\subsection{Characteristic numbers} \label{subsection_chara}
It is well-known that two closed smooth $\mathbb{Z}_2^k$-manifolds are equivariantly bordant if and only if all their equivariant Stiefel--Whitney numbers agree \cite{tomDieck1971}. Moreover, \cite[Theorem 1.4]{CLY2026} showed that, for closed smooth $\mathbb{Z}_2^k$-manifolds with isolated fixed points, it suffices to consider the characteristic numbers determined by the top Stiefel--Whitney class. Hence, for small covers with injective characteristic functions, these characteristic numbers form a complete invariant of equivariant homeomorphism.

Guillemin, Ginzburg and Karshon showed that two unitary $T^k$-manifolds with isolated fixed points are equivariantly unitary bordant if and only if their integral equivariant cohomology Chern numbers are equal~\cite[Theorem H.4]{GGK2002}.
L\"u and Wang extended this result to general unitary $T^k$-manifolds~\cite[Theorem 1.5]{LuWang2018}.
More precisely, suppose that $M^{2n}$ is a unitary $T^k$-manifold. For $1\le i\le n$, the $i$-th equivariant cohomology Chern class of $M$ is the ordinary Chern class of the vector bundle $ET^k\times_{T^k} TM$ over $M_{T^k} = ET^k\times_{T^k} M$
\begin{align*}
	c_i^{T^k}(TM) \coloneq c_i(ET^k\times_{T^k} TM)\in H^{2i}(M_{T^k}; \mathbb{Z}) = H^{2i}_{T^k}(M; \mathbb{Z}).
\end{align*}
For each partition $\omega = (i_1, \dots, i_r)$ of $|\omega| = i_1 + \dots + i_r$, denote
\begin{align*}
	c_\omega^{T^k}(TM) = c_{i_1}^{T^k}(TM)\cdots c_{i_r}^{T^k}(TM)\in H^*_{T^k}(M; \mathbb{Z}).
\end{align*}
Let $\pi_!: H^*_{T^k}(M; \mathbb{Z}) \to H^{*-2n}(BT^k; \mathbb{Z})$ be the Gysin map induced by the map $\pi: M_{T^k}\to BT^k$. The \emph{$\omega$-th equivariant cohomology Chern number} of $M$ is
\begin{align*}
	c^{T^k}_\omega [M]_{T^k} \coloneq \pi_!(c_\omega^{T^k}(TM)),
\end{align*}
which is a homogeneous polynomial of degree $2|\omega| - 2n$ in $H^*(BT^k; \mathbb{Z}) = \mathbb{Z}[x_1, \dots, x_k]$ with $|x_i| = 2$.

Hence, together with Theorem \ref{main_theorem_toric}, we can describe the isomorphism of toric manifolds in terms of equivariant cohomology Chern numbers.

\begin{proposition}
	Two toric manifolds $X_1$ and $X_2$ of complex dimension $n$ are isomorphic as varieties if and only if there exists $A\in \operatorname{GL}(n, \mathbb{Z})$ such that for any partition $\omega$, 
	\begin{align*}
		c^{T^n}_\omega [X_1]_{T^n} = (B\Phi_A)^* (c^{T^n}_\omega [X_2]_{T^n}).
	\end{align*}
\end{proposition}

\begin{proof}
	One can directly show that for any unitary $T^n$-manifold $M$ and $A\in \operatorname{GL}(n, \mathbb{Z})$,
	$$
	\begin{tikzpicture}[codi]
		\obj { |(a)| (M_{\Phi_A})_{T^n} &[2em] |(b)| M_{T^n} &[6em] |(e)| ET^n\times_{T^n} TM_{\Phi_A} &[5em] |(f)| ET^n\times_{T^n} TM \\
		       |(c)| BT^n & |(d)| BT^n & |(g)| (M_{\Phi_A})_{T^n} & |(h)| M_{T^n} \\};
		\mor[swap] a "\pi_{\Phi_A}":-> c "B\Phi_A":-> d;
		\mor a "\widetilde{B\Phi_A}":-> b \pi:-> d;
		\mor[swap] e -> g "\widetilde{B\Phi_A}":-> h;
		\mor e -> f -> h;
	\end{tikzpicture}
	$$
	are pullback diagrams, where $\widetilde{B\Phi_A}([e, x]) = [E\Phi_A(e), x]$ for $e\in ET^n$ and $x\in M_{\Phi_A}$. Thus the relation between the equivariant cohomology Chern numbers of $M$ and $M_{\Phi_A}$ is
	\begin{align} \label{Chern_number_twist}
		c_\omega^{T^n}[M_{\Phi_A}]_{T^n} = (B\Phi_A)^* (c^{T^n}_\omega [M]_{T^n}).
	\end{align}
\end{proof}

Analogously to the discussion of $\mathbb{Z}_2^k$-manifolds, the complete invariants for unitary $T^k$-manifolds can also be simplified, and it suffices to consider only a certain subset of the equivariant cohomology Chern numbers; see \cite{LuTan2011, CLY2026} for further details.

\subsection{Discussions on Hirzebruch surfaces} \label{subsection_Hirzebruch_surface}
A Hirzebruch surface $\mathscr{H}_a$ is the $\mathbb{C}P^1$-bundle over $\mathbb{C}P^1$ determined by the sheaf $\mathcal{O}\oplus \mathcal{O}(-a)$ for $a\in \mathbb{Z}$ (see \cite[Chapter V]{Hartshorne_AG} for further details). 
The Hirzebruch surface $\mathscr{H}_a$ admits an action of the algebraic torus $(\mathbb{C}^\times)^2$, making it a projective toric manifold. The rays of its associated fan $\Sigma_a$ are generated by the following vectors (see \cite[Example 3.1.16]{CLS2011})
\begin{align} \label{fan_Hirzebruch_surface}
    \begin{pmatrix}
        1\\0
    \end{pmatrix}, \quad
    \begin{pmatrix}
        0\\1
    \end{pmatrix},\quad
    \begin{pmatrix}
        -1\\a
    \end{pmatrix},\quad
    \begin{pmatrix}
        0\\-1
    \end{pmatrix}.
\end{align}
The classification of Hirzebruch surfaces under several natural equivalence relations is well-understood.

\begin{theorem} \label{class_Hirzebruch_surface}
    Let $\mathscr{H}_a$ and $\mathscr{H}_b$, $a, b\in \mathbb{Z}$, be two Hirzebruch surfaces.
    \begin{enumerate}
        \item (\cite{Hirzebruch}) $\mathscr{H}_a$ and $\mathscr{H}_b$ are isomorphic as varieties if and only if $|a| = |b|$.
        \item (\cite{BP_torictopo}) $\mathscr{H}_a$ and $\mathscr{H}_b$ are $T^2$-equivariantly homeomorphic if and only if $a = b$. Moreover, $\mathscr{H}_a$ and $\mathscr{H}_b$ are weakly $T^2$-equivariantly homeomorphic if and only if $|a| = |b|$.
        \item (\cite{BP_torictopo}) There exists a $T^2$-equivariant homeomorphism (resp. weakly $T^2$-equivariant homeomorphism) between $\mathscr{H}_a$ and $\mathscr{H}_b$ preserving omniorientations if and only if $a = b$ (resp. $|a| = |b|$).
        \item (\cite{Sarkar2012}) Every Hirzebruch surface is a $T^2$-equivariant oriented boundary.
    \end{enumerate}
\end{theorem}

Part (4) shows that equivariant oriented bordism cannot distinguish Hirzebruch surfaces.
In the following, we use equivariant cohomology Chern numbers to recover the classification results in Parts (1) and (3) in a unified manner.

Write $H^*(BT^2; \mathbb{Z}) = \mathbb{Z}[x_1, x_2]$, $|x_1|=|x_2| = 2$. 
From \eqref{fan_Hirzebruch_surface}, the four tangent weight pairs of $\mathscr{H}_a$ are 
\begin{align*}
    (x_1,\ x_2), \quad (-x_1,\ a x_1+x_2), \quad (-x_1,\ -a x_1-x_2), \quad (x_1,\ -x_2).
\end{align*}
Then by the Atiyah--Bott--Berline--Vergne (ABBV) localization formula \cite{AtiyahBott1984, BerlineVergne}, 
\begin{align*}
    c_{(1,1,1)}^{T^2}[\mathscr{H}_a]_{T^2} &= \frac{(x_1+x_2)^3}{x_1x_2} + \frac{(-x_1+a x_1+x_2)^3}{-x_1(a x_1+x_2)} + \frac{(-x_1-a x_1-x_2)^3}{-x_1(-a x_1-x_2)} + \frac{(x_1-x_2)^3}{-x_1x_2} \\
    &= -2a^2x_1 - 4ax_2.
\end{align*}

\begin{proposition} \label{Hirzebruch_surface_unibor}
    Two Hirzebruch surfaces $\mathscr{H}_a$ and $\mathscr{H}_b$, $a, b\in \mathbb{Z}$, are $T^2$-equivariantly unitary bordant if and only if $a = b$.
\end{proposition}

\begin{proof}
    If $\mathscr{H}_a$ and $\mathscr{H}_b$ are $T^2$-equivariantly unitary bordant, all their integral equivariant cohomology Chern numbers agree. Then $-2a^2x_1 - 4ax_2 = -2b^2x_1 - 4bx_2\in H^2(BT^2; \mathbb{Z})$, and so $a = b$. The converse is obvious.
\end{proof}

For $L = r_1x_1 + r_2x_2\in H^2(BT^2; \mathbb{Z})$, denote $d_L = \operatorname{gcd}(r_1, r_2)$. Since for $A\in \operatorname{GL}(2, \mathbb{Z})$, $(B\Phi_A)^*$ is a unimodular transformation on $H^2(BT^2; \mathbb{Z})$, $d_{(B\Phi_A)^*(L)} = d_L$.

\begin{proposition} \label{Hirzebruch_surface_weak}
    Two Hirzebruch surfaces $\mathscr{H}_a$ and $\mathscr{H}_b$, $a, b\in \mathbb{Z}$, are weakly $T^2$-equivariantly unitary bordant if and only if $|a| = |b|$.
\end{proposition}

\begin{proof}
    If $\mathscr{H}_a$ and $\mathscr{H}_b$ are weakly $T^2$-equivariantly unitary bordant, by \eqref{Chern_number_twist}, there exists $A\in \operatorname{GL}(2, \mathbb{Z})$ such that $c_{(1, 1, 1)}^{T^2}[\mathscr{H}_a]_{T^2} = (B\Phi_A)^* (c^{T^2}_{(1, 1, 1)} [\mathscr{H}_b]_{T^2})$. Hence, 
    \begin{align*}
        2|a|\operatorname{gcd}(|a|, 2) = d_{c_{(1, 1, 1)}^{T^2}[\mathscr{H}_a]_{T^2}} = d_{c^{T^2}_{(1, 1, 1)} [\mathscr{H}_b]_{T^2}} = 2|b|\operatorname{gcd}(|b|, 2).
    \end{align*}
    Considering the parities of $a$ and $b$, we obtain $|a| = |b|$.

    Conversely, assume $|a| = |b|$. The case $a = b$ is immediate. For $b = -a$, take
    \begin{align*}
        A=\begin{pmatrix}1&0\\0&-1\end{pmatrix}\in \operatorname{GL}(2, \mathbb{Z}),
    \end{align*}
    so that $(B\Phi_A)^*(x_1) = x_1, (B\Phi_A)^*(x_2) = -x_2$.
    The four tangent weight pairs of $\mathscr{H}_{-a}$ are $(x_1, x_2)$, $(-x_1, -a x_1+x_2)$, $(-x_1, a x_1-x_2)$, $(x_1, -x_2)$. Applying $(B\Phi_A)^*$ yields $(x_1, -x_2)$, $(-x_1, -a x_1-x_2)$, $(-x_1, a x_1 + x_2)$, $(x_1, x_2)$, which consist of the fixed-point weight data of $\mathscr{H}_a$.
    Then the ABBV localization formula implies that 
    \begin{align*}
        c_{\omega}^{T^2}[\mathscr{H}_a]_{T^2} = (B\Phi_A)^*(c_{\omega}^{T^2}[\mathscr{H}_{-a}]_{T^2})
    \end{align*}
    for any partition $\omega$. Therefore, $\mathscr{H}_a$ and $\mathscr{H}_{-a}$ are weakly equivariantly unitary bordant.
\end{proof}

Combining Theorems \ref{main_theorem_toric}, \ref{main_theorem_quasitoric}  and Corollary \ref{cor_weak} with Propositions \ref{Hirzebruch_surface_unibor} and \ref{Hirzebruch_surface_weak}, we obtain a new proof of Parts (1) and (3) of Theorem \ref{class_Hirzebruch_surface} using equivariant unitary bordism and equivariant Chern numbers.

\section*{Acknowledgments}
Zhi L\"u was partially supported by the grants from NSFC 11971112 and SIMIS-ID-2025-TP. The authors thank Tao Gong for bringing \cite{HKM2022} to their attention.


\begin{thebibliography}{CMS10b}
	
	\bibitem[AB84]{AtiyahBott1984}
	Michael~F. Atiyah and Raoul~H. Bott.
	\newblock The moment map and equivariant cohomology.
	\newblock {\em Topology}, 23(1):1--28, 1984.
	
	\bibitem[BP15]{BP_torictopo}
	Victor~M. Buchstaber and Taras~E. Panov.
	\newblock {\em Toric topology}, volume 204 of {\em Mathematical Surveys and
		Monographs}.
	\newblock American Mathematical Society, Providence, RI, 2015.
	
	\bibitem[BV83]{BerlineVergne}
	Nicole Berline and Mich\`ele Vergne.
	\newblock Z\'eros d'un champ de vecteurs et classes caract\'eristiques
	\'equivariantes.
	\newblock {\em Duke Math. J.}, 50(2):539--549, 1983.
	
	\bibitem[CLS11]{CLS2011}
	David~A. Cox, John~B. Little, and Henry~K. Schenck.
	\newblock {\em Toric varieties}, volume 124 of {\em Graduate Studies in
		Mathematics}.
	\newblock American Mathematical Society, Providence, RI, 2011.
	
	\bibitem[CLT25]{CLT2025}
	Bo~Chen, Zhi L\"{u}, and Qiangbo Tan.
	\newblock Equivariant geometric bordisms and universal complexes.
	\newblock {\em Proc. Amer. Math. Soc.}, 153(6):2687--2699, 2025.
	
	\bibitem[CLY26]{CLY2026}
	Runze Chen, Zhi Lü, and Leqi Yang.
	\newblock Reduced characteristic number criteria for ${T}^k$- and
	$(\mathbb{Z}_2)^k$-equivariant bordisms, arXiv: 2607.01889, 2026.
	
	\bibitem[CMS10a]{CMS2010_quasi}
	Suyoung Choi, Mikiya Masuda, and Dong~Y. Suh.
	\newblock Quasitoric manifolds over a product of simplices.
	\newblock {\em Osaka J. Math.}, 47(1):109--129, 2010.
	
	\bibitem[CMS10b]{CMS2010}
	Suyoung Choi, Mikiya Masuda, and Dong~Y. Suh.
	\newblock Topological classification of generalized {B}ott towers.
	\newblock {\em Trans. Amer. Math. Soc.}, 362(2):1097--1112, 2010.
	
	\bibitem[CPS10]{CPS2010}
	Suyoung Choi, Taras~E. Panov, and Dong~Y. Suh.
	\newblock Toric cohomological rigidity of simple convex polytopes.
	\newblock {\em J. Lond. Math. Soc. (2)}, 82(2):343--360, 2010.
	
	\bibitem[CS11]{CS2011}
	Suyoung Choi and Dong~Y. Suh.
	\newblock Properties of {B}ott manifolds and cohomological rigidity.
	\newblock {\em Algebr. Geom. Topol.}, 11(2):1053--1076, 2011.
	
	\bibitem[Dan78]{Danilov1978}
	Vladimir~I. Danilov.
	\newblock The geometry of toric varieties.
	\newblock {\em Uspekhi Mat. Nauk}, 33(2(200)):85--134, 247, 1978.
	
	\bibitem[Dar15]{Darby2015}
	Alastair Darby.
	\newblock Torus manifolds in equivariant complex bordism.
	\newblock {\em Topology Appl.}, 189:31--64, 2015.
	
	\bibitem[DJ91]{DavisJanus1991}
	Michael~W. Davis and Tadeusz Januszkiewicz.
	\newblock Convex polytopes, {C}oxeter orbifolds and torus actions.
	\newblock {\em Duke Math. J.}, 62(2):417--451, 1991.
	
	\bibitem[GGK02]{GGK2002}
	Victor Guillemin, Viktor Ginzburg, and Yael Karshon.
	\newblock {\em Moment maps, cobordisms, and {H}amiltonian group actions},
	volume~98 of {\em Mathematical Surveys and Monographs}.
	\newblock American Mathematical Society, Providence, RI, 2002.
	\newblock Appendix J by Maxim Braverman.
	
	\bibitem[Han05]{Hanke2005}
	Bernhard Hanke.
	\newblock Geometric versus homotopy theoretic equivariant bordism.
	\newblock {\em Math. Ann.}, 332(3):677--696, 2005.
	
	\bibitem[Har77]{Hartshorne_AG}
	Robin Hartshorne.
	\newblock {\em Algebraic geometry}, volume No. 52 of {\em Graduate Texts in
		Mathematics}.
	\newblock Springer-Verlag, New York-Heidelberg, 1977.
	
	\bibitem[Hir51]{Hirzebruch}
	Friedrich Hirzebruch.
	\newblock \"{U}ber eine {K}lasse von einfachzusammenh\"angenden komplexen
	{M}annigfaltigkeiten.
	\newblock {\em Math. Ann.}, 124:77--86, 1951.
	
	\bibitem[HKM22]{HKM2022}
	Akihiro Higashitani, Kazuki Kurimoto, and Mikiya Masuda.
	\newblock Cohomological rigidity for toric {F}ano manifolds of small dimensions
	or large {P}icard numbers.
	\newblock {\em Osaka J. Math.}, 59(1):177--215, 2022.
	
	\bibitem[HO72]{HamrickOssa}
	Gary Hamrick and Erich Ossa.
	\newblock Unitary bordism of monogenic groups and isometries.
	\newblock In {\em Proceedings of the {S}econd {C}onference on {C}ompact
		{T}ransformation {G}roups ({U}niv. {M}assachusetts, {A}mherst, {M}ass.,
		1971), {P}art {I}}, volume Vol. 298 of {\em Lecture Notes in Math.}, pages
	172--182. Springer, Berlin-New York, 1972.
	
	\bibitem[Ish12]{Ishida2012}
	Hiroaki Ishida.
	\newblock Filtered cohomological rigidity of {B}ott towers.
	\newblock {\em Osaka J. Math.}, 49(2):515--522, 2012.
	
	\bibitem[KM09]{KamishimaMasuda2009}
	Yoshinobu Kamishima and Mikiya Masuda.
	\newblock Cohomological rigidity of real {B}ott manifolds.
	\newblock {\em Algebr. Geom. Topol.}, 9(4):2479--2502, 2009.
	
	\bibitem[LT11]{LuTan2011}
	Zhi L\"u and Qiangbo Tan.
	\newblock Equivariant {C}hern numbers and the number of fixed points for
	unitary torus manifolds.
	\newblock {\em Math. Res. Lett.}, 18(6):1319--1325, 2011.
	
	\bibitem[LT14]{LuTan2014}
	Zhi L\"u and Qiangbo Tan.
	\newblock Small covers and the equivariant bordism classification of 2-torus
	manifolds.
	\newblock {\em Int. Math. Res. Not. IMRN}, 2014(24):6756--6797, 2014.
	
	\bibitem[LW18]{LuWang2018}
	Zhi L\"u and Wei Wang.
	\newblock Equivariant cohomology {C}hern numbers determine equivariant unitary
	bordism for torus groups.
	\newblock {\em Algebr. Geom. Topol.}, 18(7):4143--4160, 2018.
	
	\bibitem[Mas08]{Masuda2008}
	Mikiya Masuda.
	\newblock Equivariant cohomology distinguishes toric manifolds.
	\newblock {\em Adv. Math.}, 218(6):2005--2012, 2008.
	
	\bibitem[MS08]{MasudaSuh2008}
	Mikiya Masuda and Dong~Youp Suh.
	\newblock Classification problems of toric manifolds via topology.
	\newblock In {\em Toric topology}, volume 460 of {\em Contemp. Math.}, pages
	273--286. Amer. Math. Soc., Providence, RI, 2008.
	
	\bibitem[Oda88]{Oda1988}
	Tadao Oda.
	\newblock {\em Convex bodies and algebraic geometry}, volume~15 of {\em
		Ergebnisse der Mathematik und ihrer Grenzgebiete (3) [Results in Mathematics
		and Related Areas (3)]}.
	\newblock Springer-Verlag, Berlin, 1988.
	\newblock An introduction to the theory of toric varieties, Translated from the
	Japanese.
	
	\bibitem[Sar12]{Sarkar2012}
	Soumen Sarkar.
	\newblock {$\mathbb{T}^2$}-cobordism of quasitoric 4-manifolds.
	\newblock {\em Algebr. Geom. Topol.}, 12(4):2003--2025, 2012.
	
	\bibitem[Sto70]{Stong1970}
	Robert~E. Stong.
	\newblock Equivariant bordism and {$(\mathbb{Z}_2)^k$} actions.
	\newblock {\em Duke Math. J.}, 37(4):779--785, 1970.
	
	\bibitem[tD71]{tomDieck1971}
	Tammo tom Dieck.
	\newblock Characteristic numbers of {$G$}-manifolds. {I}.
	\newblock {\em Invent. Math.}, 13:213--224, 1971.
	
	\bibitem[Wie12]{Wiemeler2012}
	Michael Wiemeler.
	\newblock Remarks on the classification of quasitoric manifolds up to
	equivariant homeomorphism.
	\newblock {\em Arch. Math. (Basel)}, 98(1):71--85, 2012.
	
	\bibitem[Wie13]{Wiemeler2013}
	Michael Wiemeler.
	\newblock Exotic torus manifolds and equivariant smooth structures on
	quasitoric manifolds.
	\newblock {\em Math. Z.}, 273(3-4):1063--1084, 2013.
	
\end{thebibliography}
\end{document}